\documentclass[reqno,12pt]{amsart}
\usepackage{latexsym,amssymb,amsmath,epsfig}
\usepackage{amsthm}
\usepackage[latin1]{inputenc} 
\newtheorem{theorem}{Theorem}[section]
\newtheorem{lemma}{Lemma}[section]
\newtheorem{corollary}{Corollary}[section]
\newtheorem{proposition}{Proposition}

\begin{document}
\title{A new labeling construction from the $\otimes_h$-product}
\author{S. C. L\'opez}
\address{%
Departament de Matem\`{a}tiques\\
Universitat Polit\`{e}cnica de Catalunya\\
C/Esteve Terrades 5\\
08860 Castelldefels, Spain}
\email{susana.clara.lopez@upc.edu}

\author{F. A. Muntaner-Batle}
\address{Graph Theory and Applications Research Group \\
 School of Electrical Engineering and Computer Science\\
Faculty of Engineering and Built Environment\\
The University of Newcastle\\
NSW 2308
Australia}
\email{famb1es@yahoo.es}

\author{M. Prabu}
\address{British University Vietnam\\
Hanoi, Vietnam}
\email{mprabu201@gmail.com}

\maketitle

\begin{abstract}
The $\otimes_h$-product that refers the title was introduced in 2008 as a generalization of the Kronecker product of digraphs. Many relations among labelings have been obtained since then, always using as a second factor a family of super edge-magic graphs with equal order and size. In this paper, we introduce a new labeling construction by changing the role of the factors. Using this new construction the range of applications grows up considerably. In particular, we can increase the information about magic sums of cycles and crowns.
\end{abstract}

\begin{quotation}

\end{quotation}
\section{Introduction}

For the graph theory terminology and notation not defined in this paper we refer the reader to either one of the following sources \cite{BaMi,CH,G,Wa}.
However, in order to make this paper reasonably self-contained, we mention that by a $(p,q)$-graph we mean a graph of order $p$ and size $q$. For integers $m\le n$, we denote by $[m,n]$ the set $\{m,m+1,\ldots,n\}$.
In 1970, Kotzig and Rosa \cite{K1} introduced the concepts of edge-magic graphs and edge-magic labelings as follows: Let $G$ be a $(p,q)$-graph. Then $G$ is called {\it edge-magic} if there is a bijective function $f:V(G)\cup E(G)\rightarrow [1,p+q]$ such that the sum $f(x)+f(xy)+f(y)=k$ for any $xy\in E(G)$. Such a function is called an {\it edge-magic labeling} of $G$ and $k$ is called the {\it valence} \cite{K1} or the {\it magic sum} \cite{Wa} of the labeling $f$. We write $val(f)$ to denote the valence of $f$.

Motivated by the concept of edge-magic labelings, Enomoto et al. \cite{E} introduced in 1998 the concepts of super edge-magic graphs and labelings as follows: Let $f:V(G)\cup E(G)\rightarrow [1,p+q]$ be an edge-magic labeling of a $(p,q)$-graph G with the extra property that $f(v)=[1,p].$ Then G is called { \it super edge-magic} and $f$ is a {\it super edge-magic labeling} of $G$. It is worthwhile mentioning that Acharya and Hegde had already defined in \cite{AH} the concept of strongly indexable graph that turns out to be equivalent to the concept of super edge-magic graph. Although the original definitions of (super) edge-magic graphs and labelings were originally provided for simple graphs, in this paper, we understand these definitions for any graph without multiple edges. Therefore, unless otherwise specified, the graphs considered in this paper are not necessarily simple. In \cite{F2}, Figueroa-Centeno et al. provided the following useful characterization of super edge-magic simple graphs, that works in exactly the same way for graphs in general.

\begin{lemma}\label{super_consecutive} \cite{F2}
Let $G$ be a $(p,q)$-graph. Then $G$ is super edge-magic if and only if  there is a
bijective function $g:V(G)\longrightarrow [1,p]$ such
that the set $S=\{g(u)+g(v):uv\in E(G)\}$ is a set of $q$
consecutive integers.
In this case, $g$ can be extended to a super edge-magic labeling $f$ with valence $p+q+\min S$.
\end{lemma}

Let $f$ be an edge-magic labeling of a $(p,q)$-graph $G$. The {\it complementary labeling} of $f$, denoted by $\overline{f}$, is the labeling defined by the rule: $\overline{f}(x)=p+q+1-f(x)$, for all $x\in V(G)\cup E(G)$. Notice that, if $f$ is an edge-magic labeling of $G$, we have that $\overline{f}$ is also an edge-magic labeling of $G$ with valence $\hbox{val}(\overline{f}) = 3(p+q+1)-\hbox{val}(f)$. In the case of a super edge-magic labeling $f$ of a graph $G$, there is also the corresponding \textit{super edge-magic complementary labeling}, $f^{c}$, which is also super edge-magic. In this case $f^{c}$ is defined by the rule
\begin{eqnarray}
  f^{c}(x) &=& p+1-f(x), \ \forall x \in V(G),
  \end{eqnarray}
  $f^{c} (ab)$ is obtained as described in Lemma \ref{super_consecutive}, for all $ab \in E(G)$.
  Then, the valence of $f^{c}$ can be expressed in terms of the valence of $f$ as follows:
  $\hbox{val}(f^{c})=4p+q+3-\hbox{val}(f)$.
  Also, in the case that $G$ is a graph of equal order and size, new edge-magic labelings can be obtained from known
  super edge-magic labelings of $G$.
  Such labelings are known as the even and odd edge-magic labelings of $G$. The \textit{odd labeling} and the \textit{even labeling} \cite{PEM_LMR} obtained from $f$, denoted respectively by $o(f)$ and $e(f)$, are the labelings $o(f), e(f):V(G)\cup E(G)\rightarrow [1,p+q]$ defined as follows: (i) on the vertices: $o(f)(x)=2f(x)-1$ and $e(f)(x)=2f(x)$, for all $x \in V(G)$, (ii) on the edges: $o(f)(xy) = 2 \hbox{val}(f)-2p-2-o(f)(x)-o(f)(y)$ and $e(f)(xy)=2 \hbox{val}(f)-2p-1-e(f)(x)-e(f)(y)$, for all $xy \in E(G)$. The next lemma is an easy exercice.

\begin{lemma}\label{lemma: overline even=odd complementary}
Let $f$ be a super edge-magic labeling of a $(p,q)$-graph $G$. Then,
$$\overline{e(f)}= o(f_c)\hspace{0.5cm} \hbox{and} \hspace{0.5cm} \overline{o(f)}= e(f_c).$$

\end{lemma}

In \cite{F1}, Figueroa et al. defined the following product: Let $D$ be a digraph and let $\Gamma$ be a family of digraphs with the same set $V$ of vertices. Assume that $h: E(D) \to \Gamma$ is any function that assigns elements of $\Gamma$ to the arcs of $D$. Then the digraph $D \otimes _{h} \Gamma $ is defined by (i) $V(D \otimes _{h} \Gamma)= V(D) \times V$ and (ii) $((a,i),(b,j)) \in E(D \otimes _{h} \Gamma) \Leftrightarrow (a,b) \in E(D)$ and $(i,j) \in E(h(a,b))$. Note that when $h$ is constant, $D \otimes _{h} \Gamma$ is the Kronecker product.

Many relations among labelings have been established using the $\otimes_h$-product and some particular families of graphs, namely $\mathcal{S}_p$ and $\mathcal{S}_p^k$ (see for instance, \cite{ILMR,LopMunRiu1,LopMunRiu6,PEM_LMR2}).
The family $\mathcal{S}_p$ contains all super edge-magic $1$-regular labeled digraphs of order $p$ where each vertex takes the name of the label that has been assigned to it. A super edge-magic digraph $F$ is in $\mathcal{S}_p^k$ if $|V(F)|= |E(F)|=p$ and the minimum sum of the labels of the adjacent vertices is equal to $k$ (see Lemma \ref{super_consecutive}). Notice that, since each $1$-regular digraph has minimum induced sum equal to $(p+3)/2$, $ \mathcal{S}_p \subset \mathcal{S}_p^{(p+3)/2}$. The following result was introduced in \cite{LopMunRiu6}, generalizing a previous result found in \cite{F1} :

\begin{theorem} \label{theorem: spk} \cite{LopMunRiu6}
Let $D$ be a (super) edge-magic digraph and let $h: E(D) \to \mathcal{S}_p^k$ be any function. Then $D\otimes _{h} \mathcal{S}_p^k$ is (super) edge-magic.
\end{theorem}

In this paper, we characterize some relations among the induced labelings obtained from the $\otimes_h$-product, when we combine the odd and the even labelings of a particular super edge-magic labeling $f$, together with the complementary and the super edge-magic complementary constructions of the labelings involved. This is the content of Section 
\ref{section: Some labeling properties}. The main result of the paper is Theorem \ref{producte_super_k}, where, in some sense, we exchange the role of the factors established in Theorem \ref{theorem: spk}. Thus, we can enlarge the family of labeled graphs that we can obtain from the product. We conclude the paper with an application of this fact in Section \ref{section: cycles}.

\section{Some labeling properties obtained from the $\otimes_h$-product}

\label{section: Some labeling properties}
One of the research lines when we deal with edge-magic labelings of a particular graph $G$ is the study of the theoretical valences that are realizable. This problem has been completely solved for crowns of the form $ C_m \odot \overline{K}_n$,  where $m=p^k$ and $m=pq$, and $p$ and $q$ are primes (see \cite{LopMunRiu5,PEM_LMR} and \cite{LMP1}, respectively). In both cases, the proof is based in the construction of all theoretical super edge-magic valences and then, with the help of the odd and the even labelings, to complete the remaining valences. In this section we show some labeling properties in which we combine these labelings together with other labeling constructions.

The key point in the proof of Theorem \ref{theorem: spk} is to rename the vertices of $D$ and each element of $\mathcal{S}_p^k$ after the labels of their corresponding (super) edge-magic labeling $f$ and their super edge-magic labelings respectively and to define the labels of the product as follows: (i) the vertex $(a,i) \in V(D\otimes _{h} \mathcal{S}_p^k)$ receives the label: $p(a-1)+i$ and (ii) the arc $((a,i),(b,j)) \in E(D\otimes _{h} \mathcal{S}_p^k)$ receives the label: $p(e-1)+(k+p)-(i+j)$, where $e$ is the label of $(a,b)$ in D. Thus, for each arc $((a,i),(b,j)) \in E(D\otimes _{h} \mathcal{S}_p^k)$, coming from an arc $ e = (a,b) \in E(D)$ and an arc $ (i,j) \in E(h(a,b))$, the sum of labels is constant and equal to $p(a+b+e-3)+(k+p)$. That is, $p( \hbox{val}(f)-3)+k+p$. Thus, we get the next result.

\begin{lemma}\label{lemma: valenceinducedproduct} \cite{LopMunRiu6}
Let $\hat{f}$ be the (super) edge-magic labeling of the graph $D \otimes_h\mathcal{S}_p^k$ induced by a (super) edge-magic labeling $f$ of $D$. Then the valence of $\hat{f}$ is given by the formula
 \begin{eqnarray}\label{formula: valences}
\hbox{val}(\hat{f}) &=& p(\hbox{val}(f)-3) + k + p.
 \end{eqnarray}

\end{lemma}
The following proposition shows a relation among complementary labelings and induced labelings obtained from the $\otimes_h$-product.
\begin{proposition}\label{Propo: induced_edge_and complementary_edge}
Let $f$ be an edge-magic labeling of a digraph $D$. Consider any function $h: E(D)\rightarrow \mathcal{S}_p^k$. Then, there exists $\bar h: E(D)\rightarrow \mathcal{S}_p^{p+3-k}$ such that
$$
D\otimes_h\mathcal{S}_p^k\cong D\otimes_{\bar h}\mathcal{S}_p^{p+3-k}\hspace{0.5cm} \hbox{and}\hspace{0.5cm} \overline {\hat{f}}\simeq \hat {\overline{f}},$$
where $\overline {\hat{f}}$ is the complementary labeling of the induced labeling of $f$ of $D\otimes_h\mathcal{S}_p^k$ and $\hat {\overline{f}}$ is the labeling of $D\otimes_{\bar h}\mathcal{S}_p^{p+3-k}$ induced by the complementary labeling of $f$.
\end{proposition}

\begin{proof}
We will prove that the induced edge-magic labeled digraphs are isomorphic.
Let $\phi: \mathcal{S}_p^k\rightarrow \mathcal{S}_p^{p+3-k}$ be the function defined by $\phi (F)=F^c$, where $(\bar i,\bar j)\in F^c$ if and only if $(p+1-\bar i, p+1-\bar j)\in F$. Notice that the minimum induced edge sum of $F^c$ is $2p+2-(i+j)$, where $i+j$ is the maximum induced edge sum of $F$, that is, $i+j=k+(p-1)$. Thus, the minimum induced edge sum of $F^c$ is $(p+3)-k$.

Assume that $D$ is a $(n,m)$-digraph in which each vertex is identified with the label assigned to it by $f$. Then, the induced labeling $\hat{f}$ of the product $D\otimes_h\mathcal{S}_p^k$ is defined by $\hat{f}(a,i)=p(a-1)+i$, for any vertex $(a,i)\in V(D\otimes_h\mathcal{S}_p^k)$ and $\hat{f}((a,i), (b,j))=p(e-1)+k+p-(i+j)$, where $e$ is the label of $(a,b)$ assigned by $f$.
Then, since $|V(D\otimes_h\mathcal{S}_p^k)|=pn$ and $|E(D\otimes_h\mathcal{S}_p^k)|=pm$, the complementary labeling of $\hat{f}$ is defined by
\begin{itemize}
  \item[-] $\overline {\hat{f}}(a,i)=p(m+n)+1-p(a-1)-i$, for any vertex $(a,i)\in V(D\otimes_h\mathcal{S}_p^k)$ and
  \item[-] $\overline {\hat{f}}((a,i), (b,j))=p(m+n)+1-p(e-1)-k-p+(i+j)$, where $e$ is the label of $(a,b)$ assigned by $f$.
\end{itemize}

Let $\bar h=\phi\circ h: E(D)\rightarrow \mathcal{S}_p^{p+3-k}$ and consider the labeling $\bar f$ of $D$. Then the induced labeling $\hat{\overline f}$ of the product $D\otimes_{\bar h}\mathcal{S}_p^{p+3-k}$ is defined by $\hat{\overline f}(\bar a,\bar i)=p(m+n+1-a-1)+p+1-i$, that is,
\begin{itemize}
  \item[-]  $\hat {\overline{f}}(\bar a,\bar i)=p(m+n)+1-p(a-1)-i$,  for any vertex $(a,i)\in V(D\otimes_h\mathcal{S}_p^k)$ and
   \item[-] $\hat {\overline{f}}((\bar a,\bar i), (\bar b,\bar j))=p(m+n+1-e-1)+p+3-k+p-(\bar i+\bar j)$, where $e$ is the label of $(a,b)$ assigned by $f$. That is, $\hat {\overline{f}}((\bar a,\bar i), (\bar b,\bar j))=p(m+n)+1-p(e-1)-k-p+(i+j)$.

\end{itemize}

This proves the result.
\end{proof}

\begin{corollary}
Let $D$ be a (super) edge-magic digraph. Let $f$ and $\overline f$ be a (super) edge-magic labeling and its complement of $D$ respectively. Assume that $k=(p+3)/2$ and let $\hat{f}$ and $\overline{\hat{f}}$ be the edge-magic labeling and its complementary labeling of the graph $und(D \otimes_h \mathcal{S}^{p+3}_p)$ obtained from the labeling $f$ of D. Then, $$\hbox{val}(\overline{\hat{f}})= \hbox{val}(\hat{\overline{f}}).$$
\end{corollary}

\begin{proof}
It sufficies to observe that if $k=(p+3)/2$ then $p+3-k=(p+3)/2$.
\end{proof}

For digraphs $D$ with the same order and size, we obtain the next two results.
\begin{corollary}
Let $f$ be a super edge-magic labeling of a $(n,m)$-digraph $D$ with $m=n$. Consider any function $h: E(D)\rightarrow \mathcal{S}_p^k$. Then, there exists $\bar h: E(D)\rightarrow \mathcal{S}_p^{p+3-k}$ such that
$$
\overline {\widehat{o(f)}}\simeq \widehat {e(f_c)},$$
where $\overline {\widehat{o(f)}}$ is the complementary labeling of the induced labeling of $o(f)$ of $D\otimes_h\mathcal{S}_p^k$ and $\widehat {e(f_c)}$ is labeling of $D\otimes_{\bar h}\mathcal{S}_p^{p+3-k}$ induced by the even labeling of $f_c$.
\end{corollary}

\begin{proof}
Let $f$ be a super edge-magic labeling of $D$, by Proposition \ref{Propo: induced_edge_and complementary_edge}, there exists $\bar h: E(D)\rightarrow \mathcal{S}_p^{p+3-k}$ such that
$$
\overline {\widehat{o(f)}}\simeq \widehat {\overline{o(f)}},$$
where $\overline {\widehat{o(f)}}$ is the complementary labeling of the induced labeling of $o(f)$ of $D\otimes_h\mathcal{S}_p^k$ and $\widehat {\overline{o(f)}}$ is labeling of $D\otimes_{\bar h}\mathcal{S}_p^{p+3-k}$ induced by the complementary of the odd labeling of $f$. Since $\overline{o(f)}\simeq e(f_c)$, we obtain the result.

\end{proof}

With a similar proof, we obtain the next corollary.
\begin{corollary}
Let $f$ be a super edge-magic labeling of $(n,m)$-digraph $D$ with $m=n$. Consider any function $h: E(D)\rightarrow \mathcal{S}_p^k$. Then, there exists $\bar h: E(D)\rightarrow \mathcal{S}_p^{p+3-k}$ such that
$$
\overline {\widehat{e(f)}}\simeq \widehat {o(f_c)},$$
where $\overline {\widehat{e(f)}}$ is the complementary labeling of the induced labeling of $e(f)$ of $D\otimes_h\mathcal{S}_p^k$ and $\widehat {o(f_c)}$ is labeling of $D\otimes_{\bar h}\mathcal{S}_p^{p+3-k}$ induced by the odd labeling of $f_c$.
\end{corollary}

The next result is similar to Proposition \ref{Propo: induced_edge_and complementary_edge}.

\begin{proposition}\label{Propo: induced_edge_and complementary_super_edge}
Let $f$ be a super edge-magic labeling of digraph $D$. Consider any function $h: E(D)\rightarrow \mathcal{S}_p^k$. Then, there exists $\bar h: E(D)\rightarrow \mathcal{S}_p^{p+3-k}$ such that
$$
(\hat{f})_c\simeq \widehat {f_c},$$
where $(\hat{f})_c$ is the super edge-magic complementary labeling of the induced labeling of $f$ of $D\otimes_h\mathcal{S}_p^k$ and $\widehat {{f}_c}$ is the labeling of $D\otimes_{\bar h}\mathcal{S}_p^{p+3-k}$ induced by the super edge-magic complementary labeling of $f$. Moreover, val $((\hat{f})_c)=$val$(\widehat {f_c})$.
\end{proposition}

\begin{proof}
Let $\phi: \mathcal{S}_p^k\rightarrow \mathcal{S}_p^{p+3-k}$ be the function defined by $\phi (F)=F^c$, where $(\bar i,\bar j)\in F^c$ if and only if $(p+1-\bar i, p+1-\bar j)\in F$. Notice that the minimum induced edge sum of $F^c$ is $2p+2-(i+j)$, where $i+j$ is the maximum induced edge sum of $F$, that is, $i+j=k+(p-1)$. Thus, the minimum induced edge sum of $F^c$ is $(p+3)-k$.

Assume that $D$ is a $(n,m)$-digraph in which each vertex is identified with the label assigned to it by $f$. Then, the induced (super edge-magic) labeling $\hat{f}$ of the product $D\otimes_h\mathcal{S}_p^k$ is defined by $\hat{f}(a,i)=p(a-1)+i$, for any vertex $(a,i)\in V(D\otimes_h\mathcal{S}_p^k)$.
Since $|V(D\otimes_h\mathcal{S}_p^k)|=pn$, the super edge-magic complementary labeling of $\hat{f}$ is defined by
\begin{itemize}
  \item[-] $(\hat{f})_c(a,i)=pn+1-(p(a-1)+i)$, for any vertex $(a,i)\in V(D\otimes_h\mathcal{S}_p^k)$.
\end{itemize}

Let $\bar h=\phi\circ h: E(D)\rightarrow \mathcal{S}_p^{p+3-k}$ and consider the labeling $f_c$ of $D$. Then the induced labeling $\widehat{f_c}$ of the product $D\otimes_{\bar h}\mathcal{S}_p^{p+3-k}$ is defined by $\widehat{ f_c}(\bar a,\bar i)=p(\bar a-1)+\bar i$, that is,
\begin{itemize}
  \item[-]  $\widehat{ f_c}(\bar a,\bar i)=p(n-a)+p+1-i$,  for any vertex $(a,i)\in V(D\otimes_h\mathcal{S}_p^k)$
\end{itemize}

This proves the result.
\end{proof}

\section{The main result} \label{section: main}
Since the $\otimes_h$-product was first introduced in 2008 \cite{F1}, it has been proven to be an excellent technique to better understand many different types of labelings, as for instance (super) edge-magic labelings and harmonious labelings. The lack of enumerative results involving graph labelings constitutes a big gap in the literature of graph labelings that this product has helped to fill enormously. Also further applications outside the world of graph labeling have been found for the $\otimes_h$-product, as for instance it introduces new ways to construct Skolem and Langford type sequences \cite{LopMun15.a}. In summary, the $\otimes$-product constitutes a big breakthru into the world of graph labeling that allows to have a better and deeper understanding of the subject.

In all the results involving the $\otimes_h$-product, since the very beginning, it seems to be a constant to use super edge-magic labeled graphs as the second factor of the product, or at least graphs that in a way or another come from super edge-magic graphs \cite{ILMR,LopMunRiu1,LopMunRiu6}. The power of this section lies in the fact that it allows us to use other types of labeled graphs as a second factor of the product and this allows to refresh the ways of attacking old famous problems in the subject of graph labelings as we will in the next lines.

We now introduce a new family $\mathcal{T}^q_\sigma$ of edge-magic labeled graphs.
An edge-magic labeled digraph $F$ is in $\mathcal{T}^q_\sigma$ if $V(F)=V$, $|E(F)|=q$ and the magic sum of the edge-magic labeling is equal to $\sigma$.

\begin{theorem}
\label{producte_super_k}
Let $D\in \mathcal{S}_n^k$ and let
  $h$ be any function $h:E(D)\rightarrow \mathcal{T}^q_\sigma$.
Then $D\otimes_h \mathcal{T}^q_\sigma$ is edge-magic.
\end{theorem}
\proof  Let $p=|V|$. We identify the vertices of $D$ and each element of $\mathcal{T}^q_\sigma$ after the labels of their corresponding super edge-magic labeling and edge-magic labeling, respectively. Consider the following labeling of $D\otimes_h \mathcal{T}^q_\sigma$:
\begin{enumerate}
  \item If $(i,a)\in V(D\otimes_h \mathcal{T}^q_\sigma)$ we assign to the vertex the label: $$(p+q)(i-1)+a.$$
  \item If $((i,a),(j,b))\in E (D\otimes_h \mathcal{T}^q_\sigma)$ we assign to the arc the label: $$(p+q)(k+n-(i+j)-1)+(\sigma-(a+b)).$$
\end{enumerate}
Notice that, since $D\in \mathcal{S}_n^k$ is labeled with a super edge-magic labeling with minimum sum of the adjacent vertices equal to  $k$, we have $$\{(k+n)-(i+j): \ (i,j)\in E(D )\}=[1, n].$$ Moreover, since each element $F\in \mathcal{T}^q_\sigma$, it follows that $$\{(\sigma-(a+b): \ (a,b)\in E(F )\}=[1, p+q]\setminus V .$$
Thus, the set of labels in $D\otimes_h \mathcal{T}^q_\sigma$ covers all elements in $[1, n(p+q)]$. Moreover, for each arc $((i,a)(j,b))\in E (D\otimes_h \mathcal{T}^q_\sigma)$ the sum of the labels is constant and is equal to:
   $ (p+q)(k+n-3)+\sigma.$
\qed

From the previous proof, we also conclude the next result.

\begin{lemma}\label{lema: induced_valence_product_now}
Let $D \in \mathcal{S}_n^k$ and $\widetilde{h}$ be the edge-magic labeling of the graph $D\otimes_h \mathcal{T}^q_\sigma$,  induced by the super edge-magic labeling of $D$ and the function $h: E(D)\rightarrow \mathcal{T}^q_\sigma$. Then the valence of $\widetilde{h}$ is given by the formula
 \begin{eqnarray}\label{formula: valences_now}
\hbox{val}(\widetilde{h}) &=& (p+q)(k+n-3)+\sigma,
 \end{eqnarray}
where $p=|V(F)|$, for every $F\in \mathcal{T}^q_\sigma$.
\end{lemma}

\subsection{More labeling properties obtained from the $\otimes_h$-product}
Recall that, for every labeled digraph $D\in \mathcal{S}_n^k$ we can consider $D^c\in \mathcal{S}_n^{n+3-k}$, such that $D\cong D^c$, just by taking the super edge-magic complementary labeling that defines $D$.

\begin{proposition}
\label{proposition: property}
Let $D\in \mathcal{S}_n^k$ and let $h:E(D)\rightarrow \mathcal{T}^q_\sigma$ be any function.
Then, there exists $h^c: E(D^c)\rightarrow \mathcal{T}^q_{3(p+q+1)-\sigma}$ such that
$$D\otimes_{h}\mathcal{T}^q_{\sigma}\simeq D^c\otimes_{h^c}\mathcal{T}^q_{3(p+q+1)-\sigma} ,\  \hbox{and} \ \widetilde{h^c}\simeq\bar{\widetilde{h}},$$
where $\bar{\widetilde{h}}$ is the edge-magic complementary labeling of the induced labeling of $D\otimes_h\mathcal{T}^q_\sigma$ and $\widetilde{h^c}$ is the induced labeling of $D^c\otimes_{h^c}\mathcal{T}^q_{3(p+q+1)-\sigma}$.
\end{proposition}

\begin{proof}
Let $\phi: \mathcal{S}_n^k\rightarrow \mathcal{S}_n^{n+3-k}$ be the function defined by $\phi (D)=D^c$, where $(\bar i,\bar j)\in D^c$ if and only if $(p+1-\bar i, p+1-\bar j)\in D$ and $\psi: \mathcal{T}^q_\sigma\rightarrow \mathcal{T}^q_{3(p+q+1)-\sigma}$ be the function defined by $\psi (F)=\bar F$, where $(\bar i,\bar j)\in \bar F$ if and only if $(p+q+1-\bar i, p+q+1-\bar j)\in F$. 

Let $h:E(D)\rightarrow \mathcal{T}^q_\sigma$ be any function. Then, the induced edge-magic labeling $\widetilde{h}$ of the product $D\otimes_h \mathcal{T}^q_\sigma$ is defined by $\widetilde{h}(i,a)=(p+q)(i-1)+a$, for any vertex $(i,a)\in V(D\otimes_h \mathcal{T}^q_\sigma)$ and by $\widetilde{h}((i,a),(j,b))=(p+q)(k+n-(i+j)-1)+(\sigma-a-b)$, for any arc $((i,a),(j,b))\in E(D\otimes_h \mathcal{T}^q_\sigma)$.
Then, since $|V(D\otimes_h\mathcal{T}^q_\sigma)|=pn$ and $|E(D\otimes_h\mathcal{T}^q_\sigma)|=qn$, the complementary labeling $\bar{\widetilde{h}}$ of $D\otimes_h \mathcal{T}^q_\sigma$ is defined by

\begin{itemize}
  \item[-] $\bar{\widetilde{h}}(i,a)=(p+q)n+1-(p+q)(i-1)-a$, that is, $$\bar{\widetilde{h}}(i,a)=(p+q)(n+1-i-1)+(p+q+1-a),$$ for any vertex $(i,a)\in V(D\otimes_h\mathcal{T}^q_\sigma)$ and
  \item[-] $\bar{\widetilde{h}}((i,a), (j,b))=(p+q)n+1-(p+q)(k+n-(i+j)-1)-(\sigma-a-b)$, that is, 
  \begin{eqnarray*}
\bar{\widetilde{h}}((i,a), (j,b))&=&(p+q)(n+3-k+n-(n+1-i)-(n+1-j)-1)\\
&+&3(p+q+1)-\sigma-(p+q+1-a)-(p+q+1-b).
\end{eqnarray*}

 \end{itemize}

Thus, the function $h^c: E(D^c)\rightarrow \mathcal{T}^q_{3(p+q+1)-\sigma}$ defined by 
\begin{equation*}h^c(i,j)=\psi (h(n+1-i,n+1-j)),
\end{equation*} 
induces a labeling $\widetilde{h^c}$ of $D^c\otimes_{h^c}\mathcal{T}^q_{3(p+q+1)-\sigma}$, which is isomorphic to the labeling  $\bar{\widetilde{h}}$ of $D\otimes_h\mathcal{T}^q_\sigma$. Therefore, the result follows.
\end{proof}

\section{Magic sums of cycles} \label{section: cycles}
A famous conjecture of Godbold and Slater \cite{GodSla98} states that, for $n=2t+1\ge 7$ and $5t+4\le j\le 7t+5$ and for $n=2t\ge 4$ and $5t+2\le j\le 7t+1$ there is an edge-magic labeling of $C_n$, with valence $k=j$.

Let $G$ be a $(p,q)$-graph and $f:V(G)\cup E(G)\rightarrow [1,p+q]$ be a bijective function. The {\it $f$-weight} of a vertex $v\in V(G)$, $w_f(v)$, is defined to be $w_f(v)=f(v)+\sum f(e)$, where the sum is taken over all edges $e$ incident to $v$. The function $f$ is said to be a {\it vertex-magic total labeling} \cite{MacMilSlaWa}, if the vertex weight $wt_f(v)$ does not depend on $v$. It turns out, that for $2$-regular graphs the notions of edge-magic labeling and vertex-magic total labeling coincide, since we can easily obtain a vertex-magic total labeling from an edge-magic labeling and viceversa, just by translating one unit clockwise the labels: the label of each edge is assigned to one of its adjacent vertices, and  the label of the other one is assigned to the edge.

Dan McQuillan proved in \cite{McQ09} the next result that was originally stated in terms of vertex-magic total labelings.

\begin{proposition}\cite{McQ09} \label{propo: Labels from McQuilian}
Let $p$ be odd. Assume that $C_m$ has an edge-magic labeling $f$. Then,
\begin{enumerate}
  \item[(i)] $C_{pm}$ has an edge-magic labeling with valence $p(\hbox{val}(f))-3(p-1)/2$, and
  \item[(ii)] $C_{pm}$ has an edge-magic labeling with valence $3(p-1)m+\hbox{val}(f)$.
\end{enumerate}
\end{proposition}

The following structural results will be useful to prove that Proposition \ref{propo: Labels from McQuilian} can also be obtained by means of the $\otimes_h$-product. We denote by $\overrightarrow{C_n}$ and by $\overleftarrow{C_n}$ the two possible strong orientations of the cycle $C_n$, where the vertices of $C_n$ are the elements of the set $\{i\}_{i=1}^n$. It is well known that $$\overrightarrow{C_m}\otimes _h \{\overrightarrow{C_n},\overleftarrow{C_n}\}=\rm{gcd}(m,n)\overrightarrow{C}_{\rm{lcm}[m,n]}.$$

\begin{theorem}\cite{AMR} \label{cicleALI}
Let $m, n\in \mathbb{N}$ and consider the product $\overrightarrow{C}_m \otimes_h \{{\overrightarrow{C}_n, \overleftarrow{C}_n }\}$ where $h:
E(\overrightarrow{C}_m)\longrightarrow \{{\overrightarrow{C}_n, \overleftarrow{C}_n }\}$. Let $g$ be a generator of a cyclic subgroup of
$\mathbb{Z}_n$, namely $\langle{g}\rangle$, such that $|\langle{g}\rangle|=k$. Also let $N_g(h^{-})<m$ be a natural number that satisfies the congruence relation
$ m-2N_g(h^{-})\equiv g\ (mod\,\ n).$

 If the function $h$ assigns   $\overleftarrow {C}_n$ to exactly $N_g(h^{-})$ arcs of $\overrightarrow {C}_m$ then the product
$$\overrightarrow{C}_m \otimes_h \{{\overrightarrow{C}_n, \overleftarrow{C}_n }\}
$$
consists of exactly $n/k$ disjoint copies of a strongly oriented cycle $\overrightarrow {C}_{mk}$. In particular if gcd$(g,n)=1$, then
$\langle{g}\rangle=\mathbb{Z}_n$ and if the function $h$ assigns   $\overleftarrow {C}_n$ to exactly $N_g(h^{-})$ arcs of $\overrightarrow {C}_m$
then $$ \overrightarrow{C}_m \otimes_h \{{\overrightarrow{C}_n, \overleftarrow{C}_n }\}\cong\overrightarrow{C}_{mn}.$$
\end{theorem}

\begin{corollary}\label{cicleALI_coro}\cite{PEM_LMR}
Let $n\ge 3$ be an odd integer  and suppose that $m\ge 3$ is an integer such that either $m$ is odd or $m\ge n$. Then there exists a function $h:E(\overrightarrow{C_m})\rightarrow \{\overrightarrow{C_n}, \overleftarrow{C_n}\}$ such that
$$\overrightarrow{C_m}\otimes_h\{\overrightarrow{C_n},\overleftarrow{C_n}\}\cong \overrightarrow{C_{mn}}.$$
\end{corollary}

Now, by combining the previous two results and Lemmas \ref{lemma: valenceinducedproduct} and \ref{lema: induced_valence_product_now}, we obtain the next result, which, except for the technical condition in (i), it is the same result that McQuilian obtained in \cite{McQ09} (see Proposition \ref{propo: Labels from McQuilian}).

\begin{proposition} \label{propo: ALMOST_Labels from McQuilian}
Let $p$ be odd. Assume that $C_m$ has an edge-magic labeling $f$. Then,
\begin{enumerate}
  \item[(i)] $C_{pm}$ has an edge-magic labeling with valence $p(\hbox{val}(f))-3(p-1)/2$, when $m$ is odd or $m\ge p$.
  \item[(ii)] $C_{pm}$ has an edge-magic labeling with valence $3(p-1)m+\hbox{val}(f)$.
\end{enumerate}
\end{proposition}

\begin{proof}
(i) By Corollary \ref{cicleALI_coro}, there exists a function $h:E(\overrightarrow{C_m})\rightarrow \{\overrightarrow{C_p}, \overleftarrow{C_p}\}$ such that
$\overrightarrow{C_m}\otimes_h\{\overrightarrow{C_p},\overleftarrow{C_p}\}\cong \overrightarrow{C_{pm}}.$ Assume that each vertex of $C_p$ is identified by the label assigned to it by a super edge-magic labeling. Then, by Lemma \ref{lemma: valenceinducedproduct}, the induced labeling of the product $\overrightarrow{C_{pm}}$ has valence:
$\hbox{val}({\hat f})=p(\hbox{val}(f)-3)+(p+3)/2+p$, that is, $p(\hbox{val}(f))-3(p-1)/2$.

(ii) Similarly, By Theorem \ref{cicleALI}, there exists a function $h:E(\overrightarrow{C_p})\rightarrow \{\overrightarrow{C_m}, \overleftarrow{C_m}\}$ such that
$\overrightarrow{C_p}\otimes_h\{\overrightarrow{C_p},\overleftarrow{C_p}\}\cong \overrightarrow{C_{pm}}.$

Assume that each vertex of $C_p$ is identified by the label assigned to it by a super edge-magic labeling and  each vertex of $C_m$ is identified by the label assigned to it by $f$. Then, by Lemma \ref{lema: induced_valence_product_now}, the induced labeling of the product $\overrightarrow{C_{pm}}$ has valence:
$\hbox{val}({\tilde{f}})=2m((p+3)/2+p-3)+\hbox{val}(f)$, that is, $3(p-1)m+\hbox{val}(f)$. Thus, the result holds.
\end{proof}

\noindent {\bf Acknowledgements}
The research conducted in this document by the first author has been supported by the Spanish Research Council under project
MTM2014-60127-P and symbolically by the Catalan Research Council
under grant 2014SGR1147.

\end{document}